%% file: eigs_corrected.tex
\documentclass[a4paper,10pt]{amsart}
 \usepackage{amsfonts,mathrsfs}
 \usepackage{amsthm}
 \usepackage{amssymb}
\theoremstyle{definition}
\newtheorem{Def}{Definition}[section]
\newtheorem{Bsp}[Def]{Example}
\newtheorem{Bem}[Def]{Remark}
\theoremstyle{plain}

\newtheorem{Thm}[Def]{Theorem}
\newtheorem*{Thm*}{Theorem}
\newtheorem{Lem}[Def]{Lemma}
\newtheorem{Kor}[Def]{Corollary}
\newtheorem*{Kor*}{Corollary}

\newtheorem*{con*}{Conjecture}
\newtheorem*{frag*}{Question}
\newtheorem*{verm*}{Vermutung}

\include{operators}
\include{fonts}

\title[Eigenvalues of symmetric matrices]{Eigenvalues of symmetric matrices over integral domains}
\author{Mario Kummer}
\address{Universit\"at Konstanz, Germany} 
\email{Mario.Kummer@uni-konstanz.de}
\thanks{The author was supported by the Studienstiftung des deutschen Volkes.}
\begin{document}


\begin{abstract}
 Given an integral domain $A$ we consider the set of all integral elements over $A$ that can occur as an eigenvalue of a symmetric matrix over $A$. We give a sufficient criterion for being such an element. In the case where $A$ is the ring of integers of an algebraic number field this sufficient criterion is also necessary and we show that the size of matrices grows linearly in the degree of the element. The latter result settles a questions raised by Bass, Estes and Guralnick.
\end{abstract}
\maketitle
\section{Introduction}
Let $A$  be an integral domain with field of fractions $K=\Quot(A)$. Let $\bar{K}$ be the algebraic closure of $K$.
We consider the question of which elements $\theta \in \bar{K}$ are an eigenvalue of a symmetric
matrix over $A$. 
We denote this set by $\fE(A)$.
Obviously, $\fE(A)$ is a subset of the integral closure $\bar{A}$ of $A$ in $\bar{K}$.
It was shown in \cite{Bass94} that $\fE(A)$ is in fact a subring of $\bar{A}$.
Moreover, every $\theta\in \fE(A)$ must be \textit{totally real over} $A$: For every ring homomorphism $A \to R$ to a real closed field, 
every prime ideal of $A[\theta] \otimes_A R$ must be a real prime ideal, i.e. with formally real residue field. 
We denote by $A^{\rm Re}$ the set of all totally real elements over $A$ in $\bar{A}$.
Estes \cite{Estes} showed that $\fE(\Z)=\Z^{\rm Re}$.
Later Bass, Estes and Guralnick \cite{Bass94} showed that $\fE(A)=A^{\rm Re}$ in the case where $A$ is a Dedekind domain or when $-1$ is a sum of squares in $A$.

Let $A$ be any commutative ring with $1$. If $M$ is a symmetric matrix over $A$, then its characteristic polynomial has the property that its image under all homomorphisms from $A$ to a real closed field $R$ has all of its zeros in $R$. We call polynomials with this property \textit{real zero polynomials}. If furthermore for every such homomorphism all zeros of its image are simple, we call it a \textit{strict real zero polynomial}. We denote by $p(A)$ the \textit{Pythagoras number} of $A$, i.e. the smallest number $q$ such that every sum of squares in $A$ is a sum of $q$ squares in $A$. If no such number exists, let $p(A)=\infty$. Our main result is the following.
\begin{Thm*}
Let $f \in A[X]$ be a strict real zero polynomial of degree $n$. Then $f$ divides the characteristic polynomial of a symmetric matrix $M \in \Sym_r(A)$ with $r \leq (2p(A)+1)n$.
\end{Thm*}
If $A$ is an integral domain, this shows that every $\theta\in A^{\rm Re}$ which is a zero of a strict real zero polynomial is the eigenvalue of a symmetric matrix over $A$. Furthermore, if $A$ is the ring of integers of an algebraic number field, then the minimal polynomial of every $\theta \in A^{\rm Re}$ is a strict real zero polynomial. Therefore, we get the following corollary.
\begin{Kor*}
 Let $A$ be the ring of integers of an algebraic number field. Every $\theta\in A^{\rm Re}$ of degree $n$ over $A$ is an eigenvalue of a symmetric matrix over $A$ of size at most $(2p(A)+1)n$.
\end{Kor*}
Note that if $A$ is the ring of integers of an algebraic number field, then $p(A)$ is always finite \cite{Scharlau}. In the case of $A=\Z$ one has $p(\Z)=4$ by Lagrange's Four-Square Theorem. Thus, the above result answers the following question raised by Bass, Estes and Guralnick to the affirmative with $c=9$.
\begin{frag*}[\cite{Bass94}]
 Is there a constant $c$ such that every $\theta\in \Z^{\rm Re}$ is an eigenvalue 
 of a symmetric integral matrix of size at most $c\cdot \deg_{\Z}(\theta)$?
\end{frag*} 
 The main ideas used in this work were motivated by the theory of definite determinantal representations of hyperbolic polynomials in which similar questions are considered, however always over the ring of real polynomials. For a survey on this topic we refer to \cite{victorsurvey}.
  
 \section{Preliminaries}
Let $A$ always be a commutative ring with $1$.
The set of sums of squares in $A$, i.e. elements of the form $a_1^2+\ldots+a_r^2$ with $a_i \in A$, is denoted by $\Sigma A^2$.
The set of matrices over $A$ of size $m\times n$ is denoted by $\Mat_{m,n}(A)$.
The set of symmetric matrices over $A$ of size $n$ is denoted by $\Sym_n(A)$.
 We will make use of the \textit{real spectrum} of a ring. Given a ring $A$ the real spectrum $\Sper A$ is the set
of all pairs $\alpha=(\mathfrak{p},P)$ where $\mathfrak{p}$ is a prime ideal of $A$ and $P$ is an ordering of the
residue field $\kappa(\mathfrak{p})$.
A prime ideal of $A$ is a \textit{real prime ideal} if the residue class field $\kappa(\mathfrak{p})$ has at least one ordering.
A field that has at least one ordering is called \textit{formally real}.
A field $K$ is formally real if and only if $-1$ is not a sum of squares in $K$ if and only if there is a homomorphism $K \to R$ to a real closed field $R$.
Let $\alpha=(\mathfrak{p},P) \in \Sper A$ and let $\rho_{A, \mathfrak{p}}: A \to \kappa(\mathfrak{p})$ be
the canonical homomorphism.
We denote by $\Supp(\alpha)=\mathfrak{p}$ the support of $\alpha$, i.e. the prime ideal corresponding to $\alpha$.
For any element $f \in A$ we say that $f(\alpha) \geq 0$, i.e. $f$ is nonnegative in $\alpha$,
if $\rho_{A, \mathfrak{p}}(f) \in P$.
We write $f(\alpha)>0$, i.e. $f$ is positive, if $f(\alpha)\geq 0$ and $f \not\in \Supp(\alpha)$.
We say that $f \in A$ is nonnegative (positive) on $\Sper A$ if it is nonnegative (positive) in every point of $\Sper A$.
We will need the following crucial result from real algebra cf. \cite[III \S 9, Thm. 4]{ERA}.
It is also known as Stengle's Positivstellensatz.
 \begin{Thm}[Krivine's Positivstellensatz]\label{thm:krivine}
  Let $a \in A$ be positive on $\Sper A$. Then there are $s,t \in \Sigma A^2$ such that
  $(1+s) a= 1+t$.
 \end{Thm}
We refer to the book of Bochnak, Coste and Roy \cite{BCR} for proofs and more results on
the real spectrum. 

\begin{Def}
 Let $\varphi: A \to B$ be a finite ring homomorphism. We say that $\varphi$ is \textit{totally real}
 if for every homomorphism $A \to R$ to a real closed field $R$ every prime ideal of $B \otimes_A R$ is real.
 If $A$ is an integral domain, we say that $\theta\in \bar{A}$ is \textit{totally real over} $A$ if
 $A \to A[\theta]$ is totally real. The set of all elements that are totally real over $A$ in $\bar{A}$ is denoted by $A^{\rm Re}$.
\end{Def}

\begin{Bem}
 Let $K$ be a field and $\bar{K}$ its algebraic closure. Let $\theta \in \bar{K}$ and $d=[K[\theta]:K]$. The field extension $K\subseteq K[\theta]$ is totally real if and only if every ordering of $K$ has exactly $d$ distinct extensions to $K[\theta]$. This is equivalent to $\theta$ lying in every real closed subfield of $\bar{K}$. Thus, $K^{\rm Re}$ is the intersection of all real closed fields in $\bar{K}$. If $A$ is an integral domain, then it follows that $A^{\rm Re}$ is the set of all $\theta \in \bar{A}$ such that $\varphi(\theta)\in \Quot(\varphi(A))^{\rm Re}$ for every homomorphism from $\bar{A}$ to an algebraically closed field. Thus, the above definition for $A^{\rm Re}$ coincides with the definition from \cite{Bass94}. 
\end{Bem}

 \begin{Def}
 A monic polynomial $f \in A[X]$ is called a \textit{real zero polynomial}
 if for all homomorphisms $A \to R$ to a real closed field $R$
 the image of $f$ under the induced map $A[X]\to R[X]$ has all of its zeros in $R$.
 If moreover all of them are simple zeros (for all $A \to R$), then $f$ is a \textit{strict  zero polynomial}.
\end{Def}

\begin{Bsp}
 The characteristic polynomial of any $M \in \Sym_n(A)$ is a real zero polynomial.
\end{Bsp}

\begin{Bem}
 For any real zero polynomial $f \in A[X]$ the homomorphism \[A \to A[X]/(f)\] is totally real.
 Thus, roots of real zero polynomials over integral domains are totally real.
\end{Bem}

\begin{Bsp}
 If $\Sper A = \emptyset$, then every monic polynomial $f \in A[X]$ is a strict real zero polynomial.
\end{Bsp}

\begin{Bsp}
 Let $R$ be a real closed field and $A=R[z_1,\ldots,z_n]$. The polynomial $f=X^2-(z_1^2+\ldots+z_n^2) \in A[X]$
is a real zero polynomial
 but not a strict real zero polynomial: Just consider the morphism $A\to R$ that sends every $z_i$ to zero.
\end{Bsp}

\begin{Def}
 A symmetric matrix $M \in \Sym_n(A)$ is \textit{positive \mbox{(semi-)}definite}
 if for every homomorphism $\varphi: A \to R$ to a real closed field $R$
 the matrix
 $\varphi(A)$ is positive \mbox{(semi-)}definite.
\end{Def}

\begin{Bem}
 Let $f \in A[X]$ be a monic polynomial and let $B=A[X]/(f)$. As an $A$-module $B$ is free of rank $n=\deg f$.
 An $A$-basis is given by the residue classes of $1, X, \ldots, X^{n-1}$. The \textit{companion matrix} of $f$ is the representing
 matrix $C$ of the $A$-linear map $B \to B, \, h \mapsto \bar{X} \cdot h$ with respect to this basis. Its characteristic polynomial is $f$.
Let $g \in A[X]$ be a polynomial of degree at most $n-1$. Letting $Y$ be a new variable
one can write \[
               \frac{f(Y)\cdot g(X)-f(X) \cdot g(Y)}{Y-X}=\sum_{1 \leq i,j\leq n} b_{i,j} Y^{i-1} X^{j-1}
              \]
where the $b_{i,j}$ are polynomial expressions in the coefficients of $f$ and $g$.
The \textit{B\'ezout matrix} is the matrix $B(f,g)=(b_{i,j})_{1 \leq i,j \leq n}$. The determinant of the B\'ezout matrix is the resultant of $f$ and $g$, i.e.
$\det B(f,g) \in \fp$ for some prime ideal $\fp$ of $A$ if and only if the residue classes of $f$ and $g$ in $\kappa(\fp)[X]$ have a common zero in the algebraic
closure of $\kappa(\fp)$. In particular, $B(f,1)$ is invertible. Moreover, it is not hard to see that the B\'ezout matrix is symmetric and satisfies
$B(f,g_1+g_2)=B(f,g_1)+B(f,g_2)$ as well as $C B(f,g)=B(f,g) C^{ t}$. The B\'ezout matrix $B(f,f')$ where $f'$ is the derivative of $f$ is positive semidefinite
if and only if $f$ is a real zero polynomial. It is positive definite if and only if $f$ is a strict real zero polynomial. For proofs of these classical known results
and more information
we refer to \cite[\S 2.2]{Krein81}.
\end{Bem}

\section{A Positivstellensatz}\label{sec:positivstellensatz}

In this section we prove a Positivstellensatz for matrices over a ring $A$ similar to Krivine's Positivstellensatz. Note that this is an ungraded version
of the Positivstellensatz proved and used in \cite{Kum} that holds over the graded ring of real polynomials. The proof is also very similar.

\begin{Def}
 A symmetric matrix $M \in \Sym_n(A)$ is a \textit{sum of squares} if there is a matrix $Q \in \Mat_{m,n} (A)$
 such that $M=Q^{t}  Q$.
\end{Def}

\begin{Bem}
 If $M \in \Sym_n(A)$ is a sum of squares and $a \in A$, then $a^2M$ is a sum of squares as well.
\end{Bem}

\begin{Bem}
 A symmetric matrix $M \in \Sym_n(A)$ is a sum of squares 
 if and only if there are (column) vectors $v_1,\ldots, v_m \in A^n$ such that
 \[
  M=v_1  v_1^{t}+\ldots+ v_m  v_m^{t}.
 \]
 Thus the sum of matrices that are a sum of squares is again a sum of squares.
 \end{Bem}
 
 \begin{Bem}
 A matrix that is a sum of squares is  positive semidefinite.
\end{Bem}

\begin{Lem}\label{lem:matrixposaffin}
 Let $M \in \Sym_n(A)$ be positive definite.
 Then there  are $s \in \Sigma A^2$ and  $Q \in \Mat_{m,n}(A)$
 such that $(1+s) M=Q^{t} Q$
 and $m \leq n p(A)$.
\end{Lem}

\begin{proof}
  The statement is trivial if $\Sper A = \emptyset$. In that case $-1$ is a sum of squares in $A$, thus we can take $s=-1$ and $Q=0$. Therefore, we can assume that $\Sper A \neq \emptyset$.

 Let $S \subseteq A$ be the multiplicative subset of all elements of $A$ that are positive on $\Sper A$.
 In $A_S$ we can diagonalize the quadratic form represented by $M$ using the Gram--Schmidt process:
  Because $M$ is positive definite, we only have to divide by elements that are positive
 on $\Sper A$. Thus we have a $n \times n$ matrix $Q_0$ and a diagonal matrix $D_0$  with entries in $A_S$ such that
 $M=Q_0^{t} D_0 Q_0$ holds in $A_S$. 
 Clearing denominators gives an $s_1 \in S$, a matrix $Q_1$ and a diagonal matrix $D_1$,
 both with entries in $A$,
 such that
 $s_1 M=Q_1^{t} D_1 Q_1$. 
 Note that the diagonal entries of $D_1$ are in $S$. Thus by Krivine's Positivstellensatz
 there is a sum of squares $s_2 \in \Sigma A^2$ such that $s_2$ multiplied with $s_1$
 and with any entry of $D_1$ is in $1 + \Sigma A^2$.
 Every diagonal entry of $s_2 D_1$ is a sum of at most $p(A)$ squares.
 This gives the claim.
\end{proof}

\begin{Def}\label{def:sigmasummands}
 For $M \in \Sym_n(A)$ 
 we denote by $\Sigma(M)$ the subset of $A^n$ consisting of all elements $v \in A^n$ such that 
 \[s M = v^{t} v + N  
 \]
for some $s \in \Sigma A^2$ and $N \in \Sym_n(A)$ that are both a sum of squares.
\end{Def}

\begin{Lem}\label{lem:sigmaclosed}
 The subset $\Sigma(M)$ is a submodule of $A^n$.
\end{Lem}

\begin{proof}
It is clear that $A\cdot \Sigma(M) \subseteq \Sigma(M)$.
The following identity shows that $\Sigma(M)+\Sigma(M) \subseteq \Sigma(M)$:
 \[
  (v_1+v_2)^{t} (v_1+v_2) + (v_1-v_2)^{t}  (v_1-v_2) = 2 \cdot (v_1  v_1^{t} + v_2  v_2^{t}).\qedhere
 \]
\end{proof}

\begin{Bsp}
 Let $a \in A$ be positive  on $\Sper A$. Then $\Sigma(a)=A$ by Krivine's Positivstellensatz.
\end{Bsp}

\begin{Lem}\label{lem:sigma1}
 Let $M \in \Sym_n(A)$, $a \in A$, $v \in A^n$ and $b \in \Sigma(a)$. 
 If $av \in \Sigma(M)$, then $b^2 v \in \Sigma(M)$.
\end{Lem}

\begin{proof}
  There  are $s_1, r \in \Sigma A^2$
  such that $s_1  a = b^2 + r.$
 If $a  v \in \Sigma(M)$, then there are $s_2 \in \Sigma A^2$ and a sum of squares $N \in \Sym_n(A)$
 such that $s_2 M=a^2 v^{t}  v + N$.
                            Multiplying the second equation by $s_1^2$ and applying the first equation gives us
                            \begin{eqnarray*}
  s_1^2 s_2 M &=& (b^2+ r)^2  v^{t}  v + s_1^2 N \\
& =& (b^2 v)^{t} (b^2 v)+ (2b^2r+r^2) v^{t}  v+ s_1^2 N.
\end{eqnarray*}
Thus $b^2 v \in \Sigma(M)$.
\end{proof}

\begin{Kor}\label{cor:fromftoomega}
  Let $M \in \Sym_n(A)$ and $a \in (\Sigma(M): A^n)$.
 Then $b^2 \in (\Sigma(M): A^n)$ for all $b \in \Sigma(a)$.
\end{Kor}

\begin{proof}
 For all $v \in A^n$ we have $av \in \Sigma(M)$ and thus $b^2 v \in \Sigma(M)$ for all $b \in \Sigma(a)$
 by Lemma \ref{lem:sigma1}.
\end{proof}

\begin{Lem}\label{lem:norealpointinsigma}
 Let $M \in \Sym_n(A)$ be positive definite. The ideal
 $I=(\Sigma(M): A^n)$ is not contained in any real prime ideal.
\end{Lem}

\begin{proof}
 By Lemma \ref{lem:matrixposaffin} there are
 $s \in \Sigma A^2$ and  $Q \in \Mat_{m,n}(A)$
 such that $(1+s) M=Q^{t} Q$. We can assume that $m \geq n$.
 Let $J \subseteq \{1, \ldots ,m\}$ be a subset with exactly $n$ elements.
 Let $Q_{J}$ be the submatrix of $Q$ with the rows from $J$ and 
 $q_{J}=\det(Q_{J})$ its determinant.
 By the Theorem of Cauchy--Binet 
 we have that 
 \[\det((1+s) M)=\det(  Q^{t} Q)= \sum_{J \subseteq \{1, \ldots ,m\},\,\, |J|=n} q_J^2.\]
 Since $(1+s) M$ is positive definite, its determinant is positive on $\Sper A$.
 This means that the ideal generated by 
 the $q_J$'s is not contained in a real prime ideal.
 Thus it remains to show that $q_J \in I$ for every subset
 $J \subseteq \{1, \ldots ,m\}$ with $n$ elements.
 We consider the identity $Q_J^{t}  \adj(Q_J)^{t}=q_J  \textrm{I}_n$
 where $\adj(Q_J)$ is the adjugate matrix of $Q_J$. Now for any $v \in A^n$ we get by multiplying 
 $v$
 from the right to this identity that $q_J  v$ is an $A$-linear combination of the rows of $Q$
 which are in $\Sigma(M)$.
 This implies the claim.
\end{proof}

\begin{Lem}\label{lem:ifnorealppintthen}
 Let $I\subseteq A$ be an ideal that is not contained in any real prime ideal of $A$.
 There is a $f \in I$ with $\Sigma(f)=A$.
\end{Lem}

\begin{proof}
 Let $B=A/I$. Then $-1$ is positive on $\Sper B= \emptyset$, thus $-1 \in \Sigma B^2$ by Krivine's Positivstellensatz.
 Therefore, $1+s \in I$ for some $s \in \Sigma A^2$ and clearly $\Sigma(1+s)=A$.
\end{proof}

\begin{Kor}\label{cor:all}
  If  $M \in \Sym_n(A)$ is positive definite, then $(\Sigma(M):A^n)=A$ and thus $\Sigma(M)=A^n$.
\end{Kor}

\begin{proof}
 By Lemma \ref{lem:norealpointinsigma} the ideal $I=(\Sigma(M): A^n)$ is not contained in a real prime ideal.
 Thus by Lemma \ref{lem:ifnorealppintthen} 
 there is a $f \in I$ with $1 \in \Sigma(f)$.
 Thus by Corollary \ref{cor:fromftoomega} we have $1 \in (\Sigma(M):A^n)$, i.e. $(\Sigma(M):A^n)=A$.
\end{proof}

\begin{Thm}\label{cor:positivstellensatz}
 For every positive definite $M \in \Sym_n(A)$ there are $s \in \Sigma A^2$ and 
 $Q \in \Mat_{m,n}(A)$ such that $s M=\Id_n+Q^{t} Q$ and $m \leq 2 n p(A)$.
\end{Thm}

\begin{proof}
 There are $r \in \Sigma A^2$ and 
 a sum of squares $N \in \Sym_n(A)$ such that $r M=\Id_n+N$ by the preceding corollary. By Lemma \ref{lem:matrixposaffin}
 there are $r' \in \Sigma A^2$ and  $Q' \in \Mat_{m',n}(A)$
 such that $(1+r') N=Q'^{t} Q'$
 and $m' \leq n p(A)$. Thus $r (1+ r') M=\Id_n+r' \Id_n +Q'^{t} Q'$.
 Since $r'$ is a sum of at most $p(A)$ squares, $r' \Id_n +Q'^{t} Q'$ is a sum of at most $2n p(A)$ squares.
\end{proof}

\begin{Bem}
 Like Lemma \ref{lem:matrixposaffin} the statement of Theorem \ref{cor:positivstellensatz} becomes trivial if $\Sper A$ is empty because in that case $-\Id_n$ is a sum of squares.
\end{Bem}

\section{Strict real zero polynomials}
\begin{Thm}\label{lem:detrepmonic}
 Let $f \in A[X]$ be a strict real zero polynomial of degree $n$. Then $f$ divides the characteristic polynomial of a symmetric matrix $M \in \Sym_r(A)$ with $r \leq (2p(A)+1)n$.
\end{Thm}

\begin{proof}
 Let $C$ be the companion matrix of $f$ and let $B=B(f,f')$ which is positive definite. By Theorem \ref{cor:positivstellensatz} there are $s \in \Sigma A^2$ and $Q \in \Mat_{m,n}(A)$ such that $s B=I+Q^{t} Q$ and $m \leq 2 n p(A)$. Consider the matrices \[Q'=
\begin{pmatrix}
                                     \Id_n & Q^{t} \\ 0 & \Id_m
                                    \end{pmatrix} \textnormal{ and } M'= \begin{pmatrix}
                                     C & 0 \\ Q C^{t} & -Q  C^{t} Q^{t}
                                    \end{pmatrix}.
                     \]
The characteristic polynomial of $M'$ is divisible by $f$ and $Q'$ is invertible over $A$. Thus the characteristic polynomial of $M=Q'^{-1} M' Q'$ is also divisible by $f$. It is \[
              M=\begin{pmatrix}
                                     C-Q^{t} Q C^{t} & C Q^{t} \\ Q C^{t} & 0
                                    \end{pmatrix}
             \]
and it follows from $B C^{t}= C  B$ that this matrix is symmetric. The size of $M$ is $m+n\leq (2p(A)+1)n$.
\end{proof}

\begin{Bem}
 The statement of Theorem \ref{lem:detrepmonic} is in general not true if we only assume $f$ to be a real zero polynomial. In \cite[Ex. 1.5]{Bass94} they give an example of a noetherian factorial domain $A$ and an element $a\in A$ such that $f=X^2-a$ is a real zero polynomial that does not divide the characteristic polynomial of any symmetric matrix over $A$.
\end{Bem}

\begin{Kor}
 Let $A$ be an integral domain and let $\theta\in A^{\rm Re}$ be a root of a strict real zero polynomial of degree $n$. Then $\theta$ is an eigenvalue of a suitable $M \in \Sym_r(A)$ with $r \leq (2p(A)+1)n$.
\end{Kor}

As a byproduct we get that every integral element over a ring with empty real spectrum is an eigenvalue of a suitable symmetric matrix over this ring. This has already been shown in \cite[Thm. 3.8]{Bass94}.

\begin{Kor}\label{cor:sperleer}
 If $\Sper A = \emptyset$, then every element $\theta\in \bar{A}$ is an eigenvalue of a suitable $M \in \Sym_r(A)$.
\end{Kor}

\begin{Kor}
 Let $A$ be an integral domain with field of fractions $K$. Assume that no non-zero prime ideal of $A$ is real and let $f \in A[X]$ be a real zero polynomial of degree $n$. Then $f$ divides the characteristic polynomial of a symmetric matrix $M \in \Sym_r(A)$ with $r \leq (2p(A)+1)n$.
\end{Kor}

\begin{proof}
 The assumption on $A$ implies that every real zero polynomial that is irreducible over $K$ is a strict real zero polynomial. If $f$ is reducible, we can take the direct sum of the matrices that we get for each irreducible factor.
\end{proof}

\begin{Kor}\label{cor:bounds}
 Let $A$ be the ring of integers of an algebraic number field. Every  $\theta\in A^{\rm Re}$ of degree $n$ is an eigenvalue of a suitable $M \in \Sym_r(A)$ with $r \leq (2p(A)+1)n$.
\end{Kor}

\begin{proof}
Apply the preceding corollary to the minimal polynomial of $\theta$.
\end{proof}

\begin{Bem}
 The Pythagoras number of the ring of integers $A$ of an algebraic number field is always finite \cite{Scharlau}.
 Thus, for every such $A$ there is a positive integer $c$ such that
 every $\theta \in A^{\rm Re}$ is an eigenvalue of a symmetric matrix over $A$ of size at most $c \cdot \deg_A(\theta)$.
 One reason why we are not able to proof such a bound for arbitrary Dedekind domains is because in general they do not  have
 finite Pythagoras number (even if their field of fractions has finite Pythagoras number, see \cite[\S 4]{cdlr82}). Another reason is that in general not every real zero polynomial is a strict real zero polynomial.
\end{Bem}

\bigskip

 \noindent \textbf{Acknowledgements.}
I would like to thank Christoph Hanselka for drawing my attention to this problem and for helpful discussions
and remarks, especially on the proof of Lemma \ref{lem:matrixposaffin}.
I also thank the Studienstiftung des deutschen Volkes for their financial and ideal support.

\bibliographystyle{alpha}
\bibliography{eigs}
 \end{document}

%% file: operators.tex
\usepackage{mathrsfs}

\newcommand{\Mat}{\operatorname{Mat}}

\newcommand{\Id}{\operatorname{I}}
\newcommand{\Sym}{\operatorname{Sym}}

\newcommand{\Sper}{\operatorname{Sper}}
\newcommand{\Supp}{\operatorname{Supp}}

\newcommand{\Quot}{\operatorname{Quot}}

\newcommand{\adj}{\operatorname{adj}} 

%% file: fonts.tex
\newcommand{\fE}{{\mathfrak E}}

\newcommand{\fp}{{\mathfrak p}}

\newcommand{\Z}{{\mathbb Z}}




%% file: eigs_corrected.bbl
\def\cprime{$'$}
\begin{thebibliography}{CDLR82}

\bibitem[BCR98]{BCR}
Jacek Bochnak, Michel Coste, and Marie-Fran{\c{c}}oise Roy.
\newblock {\em Real algebraic geometry}, volume~36 of {\em Ergebnisse der
  Mathematik und ihrer Grenzgebiete (3) [Results in Mathematics and Related
  Areas (3)]}.
\newblock Springer-Verlag, Berlin, 1998.
\newblock Translated from the 1987 French original, Revised by the authors.

\bibitem[BEG94]{Bass94}
Hyman Bass, Dennis~R. Estes, and Robert~M. Guralnick.
\newblock Eigenvalues of symmetric matrices and graphs.
\newblock {\em J. Algebra}, 168(2):536--567, 1994.

\bibitem[CDLR82]{cdlr82}
M.~D. Choi, Z.~D. Dai, T.~Y. Lam, and B.~Reznick.
\newblock The {P}ythagoras number of some affine algebras and local algebras.
\newblock {\em J. Reine Angew. Math.}, 336:45--82, 1982.

\bibitem[Est92]{Estes}
Dennis~R. Estes.
\newblock Eigenvalues of symmetric integer matrices.
\newblock {\em J. Number Theory}, 42(3):292--296, 1992.

\bibitem[KN81]{Krein81}
M.~G. Kre{\u\i}n and M.~A. Na{\u\i}mark.
\newblock The method of symmetric and {H}ermitian forms in the theory of the
  separation of the roots of algebraic equations.
\newblock {\em Linear and Multilinear Algebra}, 10(4):265--308, 1981.
\newblock Translated from the Russian by O. Boshko and J. L. Howland.

\bibitem[KS89]{ERA}
Manfred Knebusch and Claus Scheiderer.
\newblock {\em Einf\"uhrung in die reelle {A}lgebra}, volume~63 of {\em Vieweg
  Studium: Aufbaukurs Mathematik [Vieweg Studies: Mathematics Course]}.
\newblock Friedr. Vieweg \& Sohn, Braunschweig, 1989.

\bibitem[Kum13]{Kum}
Mario Kummer.
\newblock Determinantal representations and {B}\'ezoutians.
\newblock {\em arXiv}, 1308.5560, 2013.

\bibitem[Sch80]{Scharlau}
Rudolf Scharlau.
\newblock On the {P}ythagoras number of orders in totally real number fields.
\newblock {\em J. Reine Angew. Math.}, 316:208--210, 1980.

\bibitem[Vin12]{victorsurvey}
Victor Vinnikov.
\newblock L{MI} representations of convex semialgebraic sets and determinantal
  representations of algebraic hypersurfaces: past, present, and future.
\newblock In {\em Mathematical methods in systems, optimization, and control},
  volume 222 of {\em Oper. Theory Adv. Appl.}, pages 325--349.
  Birkh\"auser/Springer Basel AG, Basel, 2012.

\end{thebibliography}
